\documentclass[11pt,a4paper,reqno]{amsart}
\pdfoutput=1
\usepackage[hidelinks,pdfstartview=FitH]{hyperref}
\usepackage[USenglish]{babel}
\usepackage{palatino,eulervm,amssymb,etoolbox,slashed,mathtools}
\usepackage[T1]{fontenc}
\usepackage{enumitem}
\setitemize{itemsep=2pt,leftmargin=2em}

\allowdisplaybreaks[4]

\newtheorem{thm}{Theorem}
\newtheorem{lemma}[thm]{Lemma}
\newtheorem{prop}[thm]{Proposition}
\newtheorem{rem}[thm]{Remark}

\newtheorem{df}[thm]{Definition}

\theoremstyle{definition}
\newtheorem{ex}[thm]{Example}
\AtEndEnvironment{ex}{\hfill\ensuremath{\Diamond}}

\newcommand{\R}{\mathbb{R}}
\newcommand{\C}{\mathbb{C}}
\newcommand{\Z}{\mathbb{Z}}
\newcommand{\T}{\mathbb{T}}
\newcommand{\tr}{\mathrm{Tr}}
\newcommand{\Cl}{\mathcal{C}\ell}

\newcommand{\inner}[1]{\left<#1\right>}
\newcommand{\id}{\mathrm{Id}}

\newcommand{\contra}{\,\lrcorner\,}
\newcommand{\D}{\slashed{D}} 
\newcommand{\dd}{\operatorname{d}}

\renewcommand{\S}{Sec.~}

\begin{document}

\title[Twisted reality and the second-order condition]{\vspace*{-1cm}Twisted reality and the second-order condition}

\author[L.~D{\k a}browski, F.~D'Andrea and A.M.~Magee]{Ludwik D{\k a}browski, Francesco D'Andrea and Adam M.~Magee}

\address[L.~D{\k a}browski]{Scuola Internazionale Superiore di Studi Avanzati (SISSA), via Bonomea 265, I-34136 Trieste}
\email{dabrow@sissa.it}

\address[F.~D'Andrea]{Universit\`a di Napoli ``Federico II'' and I.N.F.N. Sezione di Napoli, Complesso MSA, Via Cintia, 80126 Napoli, Italy}
\email{francesco.dandrea@unina.it}

\address[A.~Magee]{Scuola Internazionale Superiore di Studi Avanzati (SISSA), via Bonomea 265, I-34136 Trieste}
\email{amagee@sissa.it}



\begin{abstract}
An interesting feature of the finite-dimensional real spectral triple $(A,H,D,J)$ of the Standard Model is that it satisfies a ``second-order'' condition: conjugation by $J$ maps the Clifford algebra $\Cl_D(A)$ into its commutant, which in fact is isomorphic to the Clifford algebra itself ($H$ is a self-Morita equivalence $\Cl_D(A)$-bimodule).
This resembles a property of the canonical spectral triple of a closed oriented Riemannian manifold: there is a dense subspace of $H$ which is a self-Morita equivalence $\Cl_D(A)$-bimodule.
In this paper we argue that on manifolds, in order for the self-Morita equivalence to be implemented by a reality operator $J$, one has to introduce a ``twist'' and weaken one of the axioms of real spectral triples.
We then investigate how the above mentioned conditions behave under products of spectral triples.
\end{abstract}

\maketitle

\vspace*{-5mm}

\section{Introduction}
Spectral triples $(A, H, D)$ \cite{Con94} 
(see also \cite{CM08,GBV01,Lan02,vS15,Var06})
enrich the Gelfand-Naimark duality between spaces and their commutative algebras of functions by further encoding smoothness, calculus, and metric structure, and
allowing for a generalisation of such notions to noncommutative algebras.
In this framework, one can equip spectral triples with a real structure $J$ \cite{Con96}, which is essential to proclaiming $D$ a first-order differential operator.
Such so-called \emph{real} spectral triples have been successfully applied, for instance, to the Standard Model of particle physics. 
The spectral triple of the Standard Model is the product of the canonical real spectral triple $(A_M, H_M, D_M,J_M)$ of a spin manifold $M$ and a finite-dimensional noncommutative real spectral triple $(A_F, H_F, D_F,J_F)$ encoding the \mbox{internal} \mbox{degrees} of freedom of elementary particles.
An interesting feature is that $J_F$ also implements a ``second-order'' condition \cite{FB14a}: conjugation by $J_F$ maps the Clifford \mbox{algebra} $\Cl_{D_F}(A_F)$ into its commutant. 
In fact an even stronger property holds: the commutant of $\Cl_{D_F}(A_F)$
is isomorphic to $\Cl_{D_F}(A_F)$ itself, with $H_F$ a self-Morita equivalence $\Cl_{D_F}(A_F)$-bimodule (we call this the \emph{Hodge} property, \emph{cf.}~Def.~\ref{df:5}).
These features and their \mbox{consequences} for the example of the Standard Model were studied in \cite{DDS17,DS19,DD14} in the context of finite-dimensional spectral triples.

Similar such $\Cl_D(A)$-bimodules were investigated in great detail in the context of spectral triples of closed oriented Riemannian manifolds in \cite{LRV12}. These spectral triples, which we refer to as ``Hodge-de Rham'' spectral triples, are built on the space of complex exterior forms rather than spinors. The Dirac operator is built from the exterior derivative either as $d+d^*$ or $-i(d-d^*)$, and there is a natural reality operator associated to the Hodge star operator that intertwines these two Dirac operators, as observed already in \cite{FGR99} (see also \cite{GBV01,LRV12,Con13}).
Another natural antilinear involution is given by the Tomita-Takesaki operator, and although it does not give a real spectral triple \emph{sensu stricto} (\emph{cf.} Def.~\ref{df:realST}), it does implement a $\Cl_D(A)$ self-Morita equivalence (\emph{cf.}~\cite{LRV12}).

Motivated by the work \cite{FB14a} and by the properties of the spectral triple describing the internal degrees of freedom of particles in the noncommutative geometry approach to the Standard Model of particle physics, we are interested in spectral triples satisfying the second-order condition. Such a condition should in some sense characterise the difference between differential forms and Dirac spinors.

In \S\ref{sec:prelim}, we recall some background material about spectral triples and give a gentle review of the canonical spectral triple of a closed oriented Riemannian manifold, built from differential forms. We are particularly interested in reality operators and observe that of the two natural reality operators, the one coming from the Hodge star does not satisfy the second-order condition, whilst the one coming from the involution does not (anti)commute with the Dirac operator. The observation in \cite{FGR99} that the former real structure intertwines the two natural Dirac operators is re-interpreted in the framework of twisted reality \cite{BCDS16,BDS19}.

A natural question is then whether there exists an alternative antilinear involution giving both a real spectral triple \emph{sensu stricto} and satisfying the second-order condition. In \S\ref{sec:torus} we consider the simple example of the $2$-torus and prove that such an operator doesn't exist (\emph{cf.} Thm.~\ref{thm:main}). Twisted reality is the best one can aim for if one is interested in the second-order condition.

In light of the fact that the spectral triple of the Standard Model is a product of two factors, in \S\ref{sec:products} we study how the above mentioned conditions behave under products of spectral triples. In particular we will argue in \S\ref{sec:prodHodge} that, if one defines the tensor product of real structures in the correct way, then the Hodge property is preserved under products of real spectral triples.

\section{Spectral triples and Riemannian manifolds}\label{sec:prelim}

\subsection{Spectral triples} Here we collect some preliminary definitions \cite{Con94}.

\begin{df}
A \emph{unital spectral triple} is the collection of data $(A,H,D)$ comprising of
\begin{itemize}
\item a complex Hilbert space $H$,\footnote{We will always assume that $H$ is complex, even if $A$ may be real.}
\item a real or complex unital $*$-subalgebra $A$ of the algebra of bounded operators on $H$,
\item a self-adjoint operator $D$ on $H$ with compact resolvent,
\end{itemize}
such that $a\operatorname{Dom}(D)\subset\operatorname{Dom}(D)$ and $[D,a]$ extends to a bounded operator on $H$ for all $a\in A$. We call $D$ a (generalised) Dirac operator.
\end{df}

\noindent
The Clifford algebra of a spectral triple was introduced in \cite{FGR99} (see also \cite{LRV12}). Given a unital spectral triple we define
\begin{itemize}
\item $\Omega^1_D(A)$ to be the \emph{complex} vector subspace of $\mathcal{B}(H)$ spanned by $a[D,b]$, for $a,b\in A$; and

\item $\Cl_D(A)$ to be the \emph{complex} $C^*$-subalgebra of $\mathcal{B}(H)$ generated by $A$ and $\Omega^1_D(A)$.
\end{itemize}
We think of $\Omega^1_D(A)$ as the analogue of (smooth) differential $1$-forms and $\Cl_D(A)$ as the analogue of (continuous) sections of the Clifford algebra bundle on a compact Riemannian manifold.

A unital spectral triple $(A,H,D)$ is called \emph{even} if it is equipped with a bounded self-adjoint operator $\gamma$ on $H$ which commutes with $A$, anticommutes with $D$, and satisfies $\gamma^2=1$. We call $\gamma$ the \emph{grading} operator.

If $S$ is any subset of $\mathcal{B}(H)$ we will denote by
$$
S'\coloneqq \big\{\,\xi\in\mathcal{B}(H)\;\big\vert\;[\xi,\eta]=0,\;\forall\;\eta\in S\,\big\}
$$
its \emph{commutant}.
Given an antilinear isometry $J$ on $H$, for all $\xi\in\mathcal{B}(H)$ and $S\subset\mathcal{B}(H)$ we will further denote
\begin{align*}
\xi^\circ&\coloneqq J\xi^*J^{-1},\\
S^\circ&\coloneqq \big\{\,\xi^\circ\;\big\vert\;\xi\in S\,\big\}.
\end{align*}
Notice that the map $\xi\mapsto\xi^\circ$ is complex-linear and antimultiplicative:
$$
(\xi\eta)^\circ=\eta^\circ\xi^\circ,
$$
for all $\xi,\eta\in\mathcal{B}(H)$. If $S$ is a subalgebra of $\mathcal{B}(H)$, sometimes it will be useful to think of the above map as defining a \emph{right} action $\triangleleft$ of $S$ on $H$ by
\begin{equation}\label{eq:rightaction}
\psi\triangleleft\xi\coloneqq \xi^\circ\psi,
\end{equation}
for all $\xi\in S$ and $\psi\in H$.

Inspired by Tomita-Takesaki theory and by the example of the modular involution of a von Neumann algebra, given a unital spectral triple and an antilinear isometry $J$ we formulate the following set of conditions:
\begin{subequations}\label{eq:commutant}
\begin{align}
A^\circ &\subset A', \label{eq:commutanta}\\
\Omega^1_D(A)^\circ &\subset A', \label{eq:commutantb}\\ 
\Omega^1_D(A)^\circ &\subset \Omega^1_D(A)'. \label{eq:commutantc}
\end{align}
\end{subequations}

\begin{rem}
The above conditions can be recast in more familiar forms: condition \textup{(\ref{eq:commutanta})} means that $[a,JbJ^{-1}]=0$ for all $a,b\in A$, commonly known as the \emph{reality} or \emph{zeroth-order} condition; \textup{(\ref{eq:commutantb})} that
$\bigl[[D,a],JbJ^{-1}\bigr]=0$ for all $a,b\in A$, known as the \emph{first-order} condition;
and \textup{(\ref{eq:commutantc})} that
$\bigl[[D,a],J[D,b]J^{-1}\bigr]=0$ for all $a,b\in A$, called the \emph{second-order} condition in \cite{FB14a}.

\end{rem}

Following the example of the canonical spectral triple of a closed oriented Riemannian spin manifold \cite{Con96}, we state another set of conditions:
\begin{subequations}\label{eq:J}
\begin{align}
J^2&=\varepsilon 1, \label{eq:Ja}\\
JD&=\varepsilon' DJ, \label{eq:Jb}\\
J\gamma&=\varepsilon''\gamma J, \label{eq:Jc}
\end{align}
\end{subequations}
where $\varepsilon,\varepsilon',\varepsilon''=\pm 1$ are three signs which 
determine what is called the \emph{KO-dimension} of the spectral triple according to Table \ref{tab}. Note that \textup{(\ref{eq:Jc})} is only relevant in the even case, and that in (\ref{eq:Jb}) we are implicitly assuming that $J$ preserves the domain of $D$. 

\begin{table}[t]
$\begin{array}{|l|cccccccc|}
\hline
\text{KO-dim} & 0 & 1 & 2 & 3 & 4 & 5 & 6 & 7 \\
\hline
\hspace*{6mm}\varepsilon    & ++ & + & -+ & - & -- & - & +- & + \\
\hspace*{6mm}\varepsilon'   & +- & - & +- & + & +- & - & +- & + \\
\hspace*{6mm}\varepsilon''  & ++ &   & -+ &   & ++ &   & -- & \\
\hline
\end{array}$

\bigskip

\caption{KO-dimensions of a spectral triple. Note that in the even case there are two possible choices of antilinear isometry $J$ related by the grading $\gamma$ \cite{DD11}.}\label{tab}
\end{table}

\begin{df}\label{df:realST}
A unital (even) spectral triple $(A,H,D,(\gamma))$ is called \emph{real} if it is equipped with an antilinear isometry $J$ satisfying 
the reality \textup{(\ref{eq:commutanta})} and first-order \textup{(\ref{eq:commutantb})} conditions, and (\ref{eq:J}).  
The operator $J$ will be called the \emph{real structure} of the spectral triple.
\end{df}

Real spectral triples were introduced in \cite{Con96}. The presentation in terms of \mbox{1-forms} and the Clifford algebra can be found in \cite{FGR99}.

Notice that (\ref{eq:commutanta}) and (\ref{eq:commutantb}) are together equivalent to the condition
\begin{equation}\label{eq:realST}
\Cl_D(A)^\circ\subset A',
\end{equation}
and the three conditions \eqref{eq:commutant} are together equivalent to
\begin{equation}\label{eq:2ndorder}
\Cl_D(A)^\circ\subset\Cl_D(A)'.
\end{equation}
In view of the above considerations, we can interpret \eqref{eq:realST} by saying that $H$ is an $A$-$\Cl_D(A)$-bimodule, where the left action of $A$ is given by its inclusion as a subalgebra of $\mathcal{B}(H)$ and the right action of $\Cl_D(A)$ is given by \eqref{eq:rightaction}.
Since \eqref{eq:realST} is also equivalent to $A^\circ\subset\Cl_D(A)'$, we can also interpret it by saying that $H$ is a $\Cl_D(A)$-$A$-bimodule, where now the left action of $\Cl_D(A)$ is given by its inclusion as a subalgebra of $\mathcal{B}(H)$ and the right action of $A$ is given by \eqref{eq:rightaction}. Finally, \eqref{eq:2ndorder} can be interpreted as saying that $H$ is a $\Cl_D(A)$-$\Cl_D(A)$-bimodule.

\smallskip

One can weaken the conditions defining (real, even) spectral triples in several ways. In particular, we will need a weaker version of (\ref{eq:Jb}), which we describe as follows.
Let $\tau\in\mathcal{B}(H)$. We call $J$ a \emph{$\tau$-twisted} real structure if $J\tau$ preserves the domain of $D$ and,
instead of (\ref{eq:Jb}), satisfies the condition
\begin{equation}\label{eq:twistedJ}
\tag{\protect\ref{eq:Jb}${}^\prime$}
\tau JD=\varepsilon'DJ\tau .
\end{equation}
In the cases studied in this paper the \emph{twist} $\tau$ will also satisfy $\tau=\tau^*=\tau^{-1}$ and will commute with both $J$ and $A$. Such a structure is a special case of the twist described in \cite{BCDS16,BDS19}.

\subsection{Morita equivalence}\label{sec:Morita}
For this part we will refer to \cite{RW98}.
Given two (complex) $C^*$-algebras $A$ and $B$,
an $A$-$B$ \emph{imprimitivity bimodule} is a pair $(E,\phi)$ of a full right Hilbert $B$-module $E$ and an isomorphism $\phi\colon A\to\mathcal{K}_B(E)$ from $A$ to the $C^*$-algebra of right $B$-linear compact endomorphisms of $E$.
Two $C^*$-algebras $A$ and $B$ are called \emph{Morita equivalent} if an $A$-$B$ imprimitivity bimodule exists.\footnote{At first, M.A.~Rieffel used the name ``strong Morita equivalence'' to distinguish it from the algebraic version of this notion, but it is now customary to omit the word ``strong''. Furthermore, if $A$ and $B$ are \emph{unital} $C^*$-algebras, it is known that they are strongly Morita equivalent if and only if they are Morita equivalent as rings \cite{Bee82}.}
An $A$-$A$ imprimitivity bimodule is also called a \emph{self-Morita equivalence} bimodule.

For a more concrete example, if $V\to X$ is a complex Hermitian vector bundle over a compact Hausdorff space $X$, $C(X)$ is the $C^*$-algebra of continuous functions on $X$, and $\Gamma(\operatorname{End}(V))$ is the $C^*$-algebra of continuous sections of the endomorphism bundle of $V$, then the set of continuous sections of $V$, $\Gamma(V)$, is a $\Gamma(\operatorname{End}(V))$-$C(X)$ imprimitivity bimodule.

Every $C^*$-algebra $B$ is a self-Morita equivalence $B$-bimodule. If $B$ is unital,
every finitely generated projective right $B$-module is a full right Hilbert $B$-module (with a canonical $B$-valued inner product). If $A$ and $B$ are two unital $C^*$-algebras, every $A$-$B$ imprimitivity bimodule is finitely generated and projective (both as a left $A$-module and right $B$-module) \emph{cf.}~\cite[Ex.~4.20]{GBV01}. In particular, if $B$ is unital, every self-Morita equivalence $B$-bimodule is finitely generated and projective.

If $B$ is finite-dimensional (and therefore unital),
every finitely generated and projective right $B$-module is a finite-dimensional complex vector space. Using the structure theorem for finite-dimensional complex $C^*$-algebras it is easy to show that, conversely, every finite-dimensional complex vector space $H$ carrying a right action of $B$ is finitely generated and projective as a right $B$-module. If $A$ and $B$ are finite-dimensional, an $A$-$B$ imprimitivity bimodule is then just a pair $(H,\phi)$ of a finite-dimensional complex vector space $H$ carrying a right action $\triangleleft$ of $B$ and an isomorphism $\phi\colon A\to (\triangleleft\,B)'$ (every right $B$-linear endomorphism is adjointable and also compact in the finite-dimensional case).

\begin{df}
Let $H$ be a Hilbert space, $A,B\subset\mathcal{B}(H)$ two $C^*$-subalgebras, and $J$ an antilinear isometry on $H$. 
Regarding $H$ as a right $B$-module with right action given by \eqref{eq:rightaction}, we say that $J$ implements a Morita equivalence between $A$ and $B$, and write
$$
A\sim_J B
$$
if there exists a dense vector subspace $E\subset H$ such that the pair $(E,\phi=\id_A)$ is an $A$-$B$ imprimitivity bimodule.\footnote{Here we think of $A$ and $B$ as concrete $C^*$-algebras of bounded operators on $H$ and use $J$ to define a right action of one of the two algebras, while the isomorphism $\phi\colon A\to\mathcal{K}_B(E)$ is just the identity.
Notice that $A$ must commute with $\triangleleft\,B=B^\circ$, not with the left action of $B$. We implicitly assume that $A$ and $B^\circ$ preserve the subspace $E$.} Concretely, this means that $A=\mathcal{K}_B(E)$.
\end{df}

\begin{rem}
If $H$ is finite-dimensional, one has $A\sim_J B$ if and only if $A=(B^\circ)'$. This can be equivalently rephrased as $A'=B^\circ$ (by the double commutant theorem), as $A^\circ=B'$ (since $[a,b^\circ]=0$ if and only if $[a^\circ,b]=[b^\circ,a]^\circ=0$), or as $(A^\circ)'=B$.
\end{rem}

We can now strengthen the conditions \eqref{eq:realST} and \eqref{eq:2ndorder} by requiring that the inclusions are equalities.

\begin{df}\label{df:5}
Let $(A,H,D,J,(\gamma),(\tau))$ be a unital real (even, $\tau$-twisted) spectral triple, $\overline{A}$ the norm-closure of $A$, $\Cl_D^\gamma(A)$ the C*-algebra generated by $\Cl_D(A)$ and $\gamma$ the grading (in the even case). We will call the spectral triple:
\begin{subequations}
\begin{align}
\text{\emph{spin}} &\iff \Cl_D(A)\sim_J\overline{A} , \label{eq:spin} \\
\text{\emph{even-spin}} &\iff \Cl^\gamma_D(A)\sim_J\overline{A} , \label{eq:evenspin} \\
\text{\emph{Hodge}} &\iff \Cl_D(A)\sim_J\Cl_D(A) . \label{eq:Hodge}
\end{align}
\end{subequations}
\end{df}

\noindent
Notice that \eqref{eq:Hodge} implies \eqref{eq:2ndorder}, while \eqref{eq:spin} and \eqref{eq:evenspin} both imply \eqref{eq:realST}. Less obvious is that \eqref{eq:spin} implies \eqref{eq:evenspin}, as shown in the first part of the next proposition.

\begin{prop}\label{prop:6}
Let $(A,H,D,J,\gamma)$ be a unital real (possibly with twisted real structure) even spectral triple.
\begin{itemize}
\item[(i)] If \eqref{eq:spin} is satisfied, then $\gamma\in\Cl_D(A)$ and \eqref{eq:evenspin} is satisfied as well.

\item[(ii)] If \eqref{eq:2ndorder} is satisfied and $\gamma\in\Cl_D(A)$, then $\Omega^1_D(A)=0$.
\end{itemize}
\end{prop}

\begin{proof}
(i) Since $\gamma\in A'$, one has $\gamma^\circ\in (A^\circ)'$. But $\gamma^\circ=\pm\gamma$ due to (\ref{eq:Jc}), hence the restriction of $\gamma$ to $E$ commutes with the right action of $\overline{A}$.
If \eqref{eq:spin} is satisfied, there must exist an element $\xi\in\Cl_D(A)$ such that $\gamma-\xi$ is zero on $E$, but since it is a bounded operator, it must be zero on the whole of $H$.

\medskip

\noindent
(ii) If $\gamma\in\Cl_D(A)$ and \eqref{eq:2ndorder} is satisfied, one has $\gamma=\pm \gamma^\circ\in\Cl_D(A)'$ as well, and it follows that every $1$-form $\omega$ commutes with $\gamma$. However, all $1$-forms must also anticommute with $\gamma$, since $\gamma$ commutes with $A$ and anticommutes with $D$. Therefore we have $\omega=\frac{1}{2}\gamma\big([\gamma,\omega]+(\gamma\omega+\omega\gamma)\big)=0$.
\end{proof}

A special class of spectral triples with $\gamma\in\Cl_D(A)$ is given by the so-called orientable spectral triples. Recall that a spectral triple $(A,H,D)$ is called \emph{orientable}\footnote{According to the original terminology of \cite{Con96}.} if there is a Hochschild $n$-cycle
$$
c=\sum_{\text{finite}} a_0\otimes a_1\otimes\ldots\otimes a_n \in Z_n(A,A)
$$
for $a_0,\ldots,a_n\in A$ such that $\sum_{\text{finite}} a_0[D,a_1]\cdots[D,a_n]$ is equal to either $1$ or, in the even case, to $\gamma$.
In particular, an even orientable spectral triple has $\gamma\in\Cl_D(A)$.
Prop.\ \ref{prop:6}(ii) shows that on an orientable spectral triple satisfying the second-order (or Hodge) condition, all $1$-forms are zero.

\medskip

For real spectral triples there is another notion of orientation.
A real spectral triple is called \emph{real-orientable} if there is a Hochschild $n$-cycle \cite{Var06}
$$
c=\sum_{\text{finite}} (a_0\otimes b_0)\otimes a_1\otimes\ldots\otimes a_n \in Z_n(A,A\otimes A^\circ)
$$
for $a_0,\ldots,a_n\in A$ and $b_0\in A^\circ$ such that $\sum_{\text{finite}} a_0b_0[D,a_1]\cdots[D,a_n]$ is equal to either $1$ or, in the even case, to $\gamma$.\footnote{In \cite[\S18.1]{CM08} another version is presented where the elements $a_1,\ldots,a_n$ are required to commute with $J$.} 

In general, real-orientability is a weaker notion than orientability.
In the example of the Hodge-de Rham spectral triple (which will be described in the sections to follow), one has $JaJ^{-1}=a^*$ for all $a\in A$ so that $A=A^\circ$ and the two notions coincide. In particular, the Hodge-de Rham spectral triple reviewed in the next \S\ref{sec:HdR} is orientable if one takes the appropriate grading operator.

\subsection{Closed oriented Riemannian manifolds}\label{sec:HdR}
There are two main classes of spectral triples $(A,H,D)$ that one can associate to a closed oriented Riemannian manifold $M$. For one, $H=L^2(\bigwedge_{\C}^\bullet T^*M)$ is the Hilbert space of square-integrable (complex-valued) forms on $M$ and $D$ is the Hodge-de Rham operator. If $M$ is a spin manifold, the other is obtained by using as the Hilbert space $H$ the space of square-integrable spinors on $M$, and $D=\D$ is the Dirac operator corresponding to the spin structure. In both cases, $A=C^\infty(M)$ is the algebra of smooth functions on $M$ acting on $H$ by pointwise multiplication.

Any commutative unital spectral triple, satisfying a suitable additional set of axioms (listed for example in \cite{Con13}), turns out to be of one of these two types. More precisely, depending on the axioms, from a commutative unital spectral triple one can reconstruct either a closed oriented Riemannian manifold or a spin$^c$ manifold (see Theorems 1.1 and 1.2 in \cite{Con13}), and in the latter case the reality condition selects spin manifolds among spin$^c$. This last step follows from an algebraic characterisation of spin$^c$ manifolds in terms of Morita equivalence: such an equivalence is implemented by a reality operator $J$ exactly when the manifold is spin. This characterisation is recalled for example at the beginning of \cite{DD14} and motivates the first part of Def.~\ref{df:5}. 
In this section we spell out the construction of the spectral triple given by the Hodge-de Rham operator on differential forms and discuss some aspects related to the self-Morita equivalence of the Clifford algebra and how to implement it by means of a reality operator. We will adopt the notations and conventions of \cite{GBV01,LM89}.

\smallskip

Let $M$ be a closed oriented $n$-dimensional Riemannian manifold with metric tensor $g$. In the following, we let
\begin{equation}
A\coloneqq C^\infty(M)
\end{equation}
be the algebra of complex-valued smooth functions on $M$, $\overline{A}=C(M)$ the one of continuous functions,
$\Omega^\bullet_{\C}(M)$
the space of smooth sections of the complexified bundle of forms $\bigwedge_{\C}^\bullet T^*M\to M$, and
\begin{equation}\label{eq:EHodge}
E\coloneqq \Gamma({\textstyle\bigwedge^\bullet_{\C}}T^*M)
\end{equation}
the $C(M)$-module of continuous sections.
The Riemannian metric $g$ induces a Hermitian product on the fibres of the bundle $\bigwedge_{\C}^\bullet T^*M$, and a $C(M)$-valued Hermitian product on $E$, given by (see \emph{e.g.},~\S9.B of \cite{GBV01})
$$
(\eta,\xi)\coloneqq \det\bigl[g^{-1}(\overline{\eta}_i,\xi_j)\bigr] 
$$
for all products of $1$-forms $\eta=\bigwedge_{i=1}^k\eta_i$ and $\xi=\bigwedge_{j=1}^k\xi_j$, which is extended to $E$ by linearity and by declaring that forms with different degree are mutually orthogonal.
With the above Hermitian structure, $E$ becomes a full right Hilbert $\overline{A}$-module (like any module of continuous sections of a Hermitian vector bundle on $M$).

We let $H\coloneqq L^2(\bigwedge^\bullet_{\C}T^*M)$ be the Hilbert space completion of $E$ with respect to the inner product
$$
\inner{\eta,\xi}\coloneqq \int_M(\eta,\xi)\,\omega_g ,
$$
where $\omega_g$ is the Riemannian volume form, given on any positively oriented chart by $\omega_g=\sqrt{\det(g)}dx^1\wedge\ldots\wedge dx^n$.

The Hodge star operator $\star$ on real-valued $k$-forms is implicitly defined by the equality
$$
\eta\wedge (\star\,\xi)=(\eta,\xi)\omega_g
$$
for all real $k$-forms $\eta$ and $\xi$, and satisfies the well-known relations
$$
\star^2=(-1)^{k(n-k)}$$ on $k$-forms, and $${\star}\,\circ\,\dd\,\circ\;{\star}={}(-1)^{n(k+1)+1}\dd^*,
$$
where $\dd$ is the exterior derivative on (smooth) forms and $\dd^*$ is its formal adjoint.
The Hodge star operator can be extended to complex-valued forms linearly, \emph{e.g.},~as in \cite{Wel80}, or antilinearly, \emph{e.g.},~as in \cite{Sch63}. We adopt the first convention.

\begin{prop}[{\protect\cite{GBV01,LM89}}]\label{prop:HdR}
With $A=C^\infty(M)$ and $H=L^2(\bigwedge^\bullet_{\C}T^*M)$ as above, and with the closure of the (essentially self-adjoint elliptic \cite{LM89}) operator $\dd+\dd^*$ (the Hodge-de Rham operator of $(M,g)$), we get a unital even spectral triple $(A,H,\dd+\dd^*)$.
\end{prop}

Two natural gradings $\gamma$ and $\chi$ are given on $k$-forms by
$$
\gamma\coloneqq (-1)^k
$$
and (see \emph{e.g.},~the proof of Theorem 11.4 in \cite{Con13})
$$
\chi\coloneqq i^{-\frac{n(n+1)}{2}}(-1)^{k(n-k)+\frac{k(k+1)}{2}}\star .
$$
Evidently $\dd+\dd^*$ anticommutes with $\gamma$ and, since $\chi\circ \dd\circ\:\chi=(-1)^{n+1}\dd^*$, if $n$ is even $\chi$ also anticommutes with $\dd+\dd^*$.

If $n$ is odd, since $\dd+\dd^*$ and $\chi$ commute one can use this grading to reduce the Hilbert space and build a spectral triple on the eigenspace of $\chi$ with eigenvalue $+1$ (which is what is done \emph{e.g.},~in \cite{Con13}). We will not follow this approach, since we want $H$ to be isomorphic to the space of sections of the Clifford algebra bundle (up to a completion) in both the odd- and even-dimensional cases.

\medskip

In order to talk about the Hodge condition, we must first recall the (geometric) definition of the Clifford algebra bundle $\Cl_{\C}(M,g)\to M$. The fibre at a point $p\in M$ is the unital associative complex algebra generated by vectors $v,w\in T^*_pM$ under the relation $v\cdot w+w\cdot v=2g^{-1}(v,w)$.\footnote{Here we use a different sign convention than in \S\ref{sec:torus} to adapt to the convention in \cite{GBV01}.
The choice of sign in front of $g$ is in any case inessential, since the \emph{complexified} Clifford algebra is invariant under the replacement $g\to -g$ (up to isomorphism).} Denoting by
$\contra$ the interior product, which is the adjoint of the (left) exterior product, a left action $\lambda$ and a right action $\rho$ (an anti-representation) of the algebra $\Gamma(\Cl_{\C}(M,g))$ on forms are given on each fibre by
$$
\lambda(v)w\coloneqq v\wedge w+v\contra w\qquad\text{and}\qquad
\rho(v)w\coloneqq (-1)^k(v\wedge w-v\contra w)
$$
for all $v\in T_p^*M$ and $w\in\bigwedge^k_{\C}T^*_pM$.\footnote{We follow the conventions of \cite[\S5.1]{GBV01}. In particular, $\lambda(v)^2=g^{-1}(v,v)$ with a $+$ sign, in contrast to the $-$ sign in \cite[Proof of Thm.~11.4]{Con13}.}
We will refer to $\lambda$ and $\rho$ as left and right Clifford multiplication.
They turn the space $E$ in \eqref{eq:EHodge} into a $\Gamma(\Cl_{\C}(M,g))$-bimodule.
One can show that (see \cite{Con13} or \cite{LM89}) the grading $\chi$ is given at each point by\footnote{Globally, it is proportional to left Clifford multiplication by the Riemannian volume form. Notice the different phase compared to \cite{Con13}, due to our different sign conventions.}
$$
\chi=i^{-\frac{n(n-1)}{2}}\lambda(e^1e^2\cdots e^n),
$$
where $(e^i)_{i=1}^n$ is a positively oriented orthonormal basis of $T^*_pM$.

A vector bundle isomorphism $\Cl_{\C}(M,g)\to\bigwedge^\bullet_{\C}T^*M$ is given on each fibre by
\begin{equation}\label{eq:orthosigma}
e^{i_1}e^{i_2}\cdots e^{i_k}\mapsto e^{i_1}\wedge e^{i_2}\wedge\ldots\wedge e^{i_k},
\end{equation}
for $1\leq i_1<i_2<\ldots<i_k\leq n$, or more generally by
$$
\sigma\colon v_1v_2\cdots v_k\mapsto\lambda(v_1)\lambda(v_2)\cdots\lambda(v_k)1
$$
for all $v_1,\ldots,v_k\in T_p^*M$. 
The inverse map \cite[Eq.~(5.4)]{GBV01}
\begin{equation}\label{eq:Qmap}
Q(v_1\wedge\ldots\wedge v_k)=\frac{1}{k!}\sum_{\pi\in\mathcal{S}_k}(-1)^{\|\pi\|}v_{\pi(1)}\cdots v_{\pi(k)}
\end{equation}
will be useful later on. The maps $\sigma$ and $Q$ are called the \emph{symbol map} and the \emph{quantisation map} respectively in \cite{GBV01}.

These maps give a vector space isomorphism $\Gamma(\Cl_{\C}(M,g))\to E$ on sections intertwining the left/right multiplication of the algebra 
$\Gamma(\Cl_{\C}(M,g))$ on itself with the left/right Clifford multiplication on $E$. The next proposition then follows.

\begin{prop}
$E$ is a self-Morita equivalence $\Gamma(\Cl_{\C}(M,g))$-bimodule.
\end{prop}

An involution on sections of the Clifford algebra bundle is defined as follows. At each point $p$, on products of real cotangent vectors we have
$$
(v_1v_2\cdots v_k)^*\coloneqq v_k\cdots v_2v_1
$$
for all $v_1,\ldots,v_k\in T_p^*M$. The map is then extended antilinearly to the fibre of $\Cl_{\C}(M,g)$ at $p$, and pointwise to the algebra of continuous sections. The left Clifford action then transforms this involution into the adjoint operation.

\begin{lemma}
One has $$\lambda(\xi)^*=\lambda(\xi^*)$$ for all $\xi\in\Gamma(\Cl_{\C}(M,g))$.
\end{lemma}
\begin{proof}
Both sides of the equality are antilinear antihomomorphisms. It is enough to prove the equality for generators, which means $\lambda(v)^*=\lambda(v)$
for all real cotangent vectors $v\in T_p^*M$ and all $p\in M$. This immediately follows from the definition of the interior product as the adjoint of the left wedge product: $v\contra=(v\,\wedge )^*$.
\end{proof}

Let $\omega=\omega_1\wedge\omega_2\wedge\ldots\wedge\omega_k$ be a product of $k$ $1$-forms. Two natural antilinear isometries on forms $C_1$ and $C_2$ are given by pointwise complex conjugation
\begin{equation}\label{eq:C1}
C_1\omega=\overline{\omega},
\end{equation}
and its composition with the canonical anti-involution, explicitly,
\begin{align}
C_2\omega&=\overline{\omega}_k\wedge\overline{\omega}_{k-1}\wedge\ldots\wedge\overline{\omega}_1\nonumber\\
&=(-1)^{k(k-1)/2}\overline{\omega}.\label{eq:C2}
\end{align}
We will adopt the notation $\omega^*\coloneqq C_2(\omega)$. We can define an antilinear isomorphism $\xi\mapsto\overline{\xi}$ on the Clifford algebra as well, by declaring it to be the identity on real cotangent vectors \cite[Ex.~5.5]{GBV01}. The symbol map and quantisation map intertwine the two complex conjugations as $Q(\overline{\omega})=\overline{Q(\omega)}$ for all forms $\omega$. One can also check on a basis \eqref{eq:orthosigma} that $\sigma(\xi^*)=\sigma(\xi)^*$, so that the symbol map and quantisation map intertwine the main anti-involutions as well.

\begin{rem}
It is well-known and straightforward to check that, equipped with the antilinear isometry $C_1$, $(A,H,D,C_1)$ is a \emph{real} spectral triple.\footnote{The reality operator usually called \emph{charge conjugation} is given by $\gamma C_1$ \cite[Ex.~5.6]{GBV01}.}
\end{rem}

\begin{lemma}\label{lemma:21}
Let $J$ be the antilinear isometry on $H$ given by
\begin{equation}\label{eq:Jonforms}
J\omega \coloneqq  \omega^* .
\end{equation}
Then 
\begin{enumerate}[label=(\roman*)]
\item for all sections $\xi$ of the Clifford algebra bundle,
$$
J\,\lambda(\xi)\,J^{-1}=\rho(\xi);
$$
\item $J\circ \dd\circ J^{-1}=\dd\circ\gamma$.
\end{enumerate}
\end{lemma}

\begin{proof}
Note that $J^{-1}=J$ and that $J\coloneqq \sigma\circ *\circ Q$.
Since the main anti-involution on the Clifford algebra exchanges left and right multiplication, the corresponding operator on forms intertwines the left and right Clifford actions. 

Let $\omega$ be a $k$-form. Then
\begin{align*}
\dd J\omega&=(-1)^{k(k-1)/2}\dd\overline{\omega}=
(-1)^{k(k-1)/2}\overline{\dd\omega} \\
&=(-1)^{k(k-1)/2}(-1)^{k(k+1)/2}J(\dd\omega)=(-1)^{k^2}J(\dd\omega)=J(\dd(\gamma\omega)) ,
\end{align*}
where last equality follows from the observation that $k^2$ and $k$ have the same parity.
\end{proof}

Since $\dd$ and $\gamma$ anticommute, it follows from previous lemma that
$$
J\circ \dd^*\!\circ\:J^{-1}=(\dd\circ\:\gamma)^*=-\dd^*\!\circ\:\gamma .
$$
If we denote by $D$ the closure of the operator $-i(\dd-\dd^*)$ (called the Hodge-Dirac operator in \cite[Def.~9.24]{GBV01}), then $D$ and the Hodge-de Rham operator $d+d^*$ of Prop.~\ref{prop:HdR} are related by the operator $J$:
$$
J(\dd+\dd^*)J^{-1}=iD\gamma .
$$
Note that the definition of a noncommutative manifold in \cite{FGR99} is motivated by a similar observation, see \cite[p.~102]{FGR99}.

We now adopt $D$ as our Dirac operator and relate the geometric and algebraic definitions of Clifford algebra. We know that $D$ is a Dirac-type operator, given on smooth sections by \cite[p.~426]{GBV01}
$$
D=-i(\dd-\dd^*)=-i\,\lambda\circ\nabla ,
$$
where $\nabla\colon\Omega_{\C}^\bullet(M)\to \Omega^{\bullet+1}_{\C}(M)$ is the Levi-Civita connection.\footnote{Similarly, $(\dd+\dd^*)\gamma=\rho\circ\nabla$.}
In particular, it follows from the Leibniz rule that
\begin{equation}\label{eq:DfHodge}
i[D,f]=\lambda(\dd f) 
\end{equation}
for all $f\in C^\infty(M)$.

\begin{prop}\label{prop:CLDAhodge}
We have $\Cl_D(A)=\Gamma(\Cl_{\C}(M,g))$, which acts on $H$ via $\lambda$.
\end{prop}
\begin{proof}
It follows from \eqref{eq:DfHodge} that the (algebraic) Clifford algebra is the $C^*$-algebra of bounded operators on $H$ generated by smooth functions and Clifford multiplication by $1$-forms. In other words, $\Cl_D(A)$ is generated by smooth sections of the Clifford algebra bundle. The algebra $\Gamma(\Cl_{\C}(M,g))$ on the other hand is generated by continuous functions and continuous sections of $T^*M$ (as one can show by using a partition of unity subordinated to a finite open cover of $M$,
which exists since $M$ is compact). Every continuous function is the norm limit of a sequence of smooth functions. Also every continuous section $\xi$ of $T^*M$ is the norm limit of a sequence of smooth sections, where the norm $\|\lambda(\xi)\|$ is the operator norm on $H$ composed with $\lambda$.
Indeed, for $\xi\in\Gamma(T^*M)$ one has $\lambda(\xi)^*\lambda(\xi)=g^{-1}(\xi,\xi)$ and then
\begin{equation}\label{eq:provaxi}
\|\lambda(\xi)\|^2=\|\lambda(\xi)^*\lambda(\xi)\|=\sup_{p\in M}|\xi(p)|^2 ,
\end{equation}
where the norm on the right hand side is the one on $T_p^*M$ coming from the Riemannian metric.

Now, any continuous $\xi$ can be written as a finite sum of continuous sections each supported on a chart (by using a partition of unity). To conclude the proof it is then enough to show that continuous sections supported on a chart are the norm limit of smooth sections supported on a chart. But this follows trivially from \eqref{eq:provaxi} and the fact that in a chart, sections of the Clifford algebra bundle look like matrices of (continuous/smooth) functions.
\end{proof}

It follows from Prop.~\ref{prop:CLDAhodge} that, for the spectral triple considered here, the dense subspace $E$ of $H$ in \eqref{eq:EHodge} is a self-Morita equivalence $\Cl_D(A)$-bimodule. By Lemma \ref{lemma:21} the main anti-involution $J$ exchanges the left and right Clifford multiplication, so that we can finally make the following claim.

\begin{prop}
The data
$$
(A,H,D,J,\gamma)=\big(C^\infty(M),L^2({\textstyle\bigwedge_{\C}^\bullet T^*M}),-i(\dd-\dd^*),J,\gamma\big) ,
$$
with $J$ given by \eqref{eq:Jonforms} and $\gamma$ given by the degree of forms, comprise an even spectral triple with $\tau$-twisted real structure satisfying the Hodge condition \eqref{eq:Hodge}. The twist $\tau$ is given by $(-1)^{k(k+1)/2}$ times the identity operator on $k$-forms. The KO-dimension of this spectral triple is $0\bmod8$.
\end{prop}

\begin{proof}
The statement about the Hodge condition follows from the discussion above. Clearly $J^2=1$, so that (\ref{eq:Ja}) is satisfied with sign $\varepsilon=+1$. Since $J$ doesn't change the degree of a form, (\ref{eq:Jc}) is also satisfied with sign $\varepsilon''=+1$. Finally, $D$ anticommutes with $C_1$ (since $\dd\overline{\omega}=\overline{\dd\omega}$ for all forms $\omega$). If $\omega$ has degree $k$, then
$$
JDJ^{-1}\omega=(-1)^{\frac{k(k-1)}{2}}JDC_1\omega=(-1)^{\frac{k(k-1)}{2}+\frac{k(k+1)}{2}}C_1DC_1\omega=(-1)^{k+1}D\omega ,
$$
where we used the fact that $k^2$ and $k$ have the same parity. It follows that
\begin{align*}
\tau JD\omega=(-1)^{\frac{(k+1)(k+2)}{2}}JD\omega&=(-1)^{\frac{(k+1)(k+2)}{2}+(k+1)}DJ\omega \\
&=(-1)^{2(k+1)^2}DJ\tau\omega=DJ\tau\omega .
\end{align*}
Thus, \eqref{eq:twistedJ} is satisfied with sign $\varepsilon'=+1$.
\end{proof}

It is worth mentioning that, although in the reconstruction theorem (\emph{cf.}\ \cite{Con13}) the role of $J$ is to pass from spin$^c$ structures to spin, the presence of a reality operator is crucial even for spectral triples that are not built from Dirac spinors. In the noncommutative case, it is necessary to make sense of the condition which makes the Dirac operator a first-order differential operator, and in the noncommutative approach to gauge theories, it is required in order to define the adjoint action of the gauge group, the real part of the spectral triple, and more. Furthermore, differential forms can be constructed from spinors by the use of twisted modules \cite[\S2.5]{Var06}. It is therefore not unreasonable to investigate real structures even in the absence of (explicit) spinors.

The interest in self-Morita equivalences of a Clifford algebra (which we note are implemented by $J$) is mainly motivated by the fact that this is what happens for the finite-dimensional part of the spectral triple of the Standard Model \cite{DDS17,DS19}. It is an interesting observation, coming from $SU(5)$ GUTs, that the Hilbert space of such a spectral triple (when considering only one generation of particles) is isomorphic to the exterior algebra $\bigwedge^5\C$ (with the representation of the gauge group the restriction of the natural representation of $SU(5)$ on such a space). It is tempting then to speculate that the Hilbert space of the spectral triple of the Standard Model might have as much to do with differential forms as Dirac spinors, though any further investigation lies outside the scope of this paper.

\section{The Hodge-de Rham operator on the torus}\label{sec:torus}
In this section we will specialise the discussion to the Hodge-de Rham spectral triple on the $2$-torus $\T^2\coloneqq \R^2/\Z^2$. We use this example to argue that, when considering the Hodge-de Rham spectral triple of a closed oriented Riemannian manifold, the second-order condition \eqref{eq:2ndorder} is incompatible with (\ref{eq:Jb}) and, if one wants to enforce \eqref{eq:2ndorder}, then (\ref{eq:Jb}) must be replaced by \eqref{eq:twistedJ}.

We think of functions/forms on $\T^2$ as $\Z^2$-invariant functions/forms on $\R^2$ and use Pauli matrices denoted by:
$$
\sigma_1\coloneqq \bigg(\!\begin{array}{rr} 0 & \,1 \\ 1 & 0\end{array}\!\bigg),\qquad
\sigma_2\coloneqq \bigg(\!\begin{array}{rr} 0 & \!-i \\ i & 0\end{array}\!\bigg),\qquad
\sigma_3\coloneqq \bigg(\!\begin{array}{rr} 1 & 0 \\ 0 & \!-1\end{array}\!\bigg),
$$
and denoting by $\sigma_0$ the $2\times2$ identity matrix. Consider the isomorphism of complex vector spaces from complex differential forms on $\T^2$ to $2\times 2$ matrices of complex functions given by
\begin{equation}\label{eq:T}
\begin{gathered}
T\colon\Omega^\bullet_{\C}(\T^2) \to M_2(C^\infty(\T^2)),\\
T\colon f_0+f_1dx +f_2dy+f_3dx\wedge dy \mapsto\sigma_0f_0+i\sigma_1f_1+i\sigma_2f_2-i\sigma_3f_3
\end{gathered}
\end{equation}
for $f_0,\ldots, f_3\in C^\infty(\T^2)$.
Under this isomorphism,
the natural inner product of forms (associated to the flat metric on $\T^2$)
becomes the natural inner product on matrices of functions
$$
\inner{a,b}=\frac{1}{2}\int_0^1dx\int_0^1dy\,\tr(a^*b) ,
$$
and the Hilbert space completion is $H\coloneqq M_2(L^2(\T^2))$. 
The Hodge-de Rham operator is mapped to the operator
$$
D\coloneqq T\circ (\dd+\dd^*)\circ T^{-1}=iL_{\sigma_1}\frac{\partial}{\partial x}+iL_{\sigma_2}\frac{\partial}{\partial y},
$$
where $L$ and $R$ denote respectively left and right pointwise matrix multiplication:
$$
L_ab\coloneqq ab\qquad\text{and}\qquad R_ab\coloneqq ba,
$$
for all $a\in M_2(L^\infty(\T^2))$ and $b\in M_2(L^2(\T^2))$.
If $f$ is a scalar function, we identify $f$ with $f\sigma_0$ and $L_f=R_f$ will be denoted simply by $f$. 
The natural grading given by the degree of forms is transformed by $T$ into the operator $\gamma$ on $H$ given by
$$
\gamma a\coloneqq \sigma_3a\sigma_3 ,
$$
for all $a\in H$. Lastly, we set $A\coloneqq C^\infty(\T^2)$.

\begin{rem}
With the notation above, $(A,H,D,\gamma)$ is a spectral triple unitarily equivalent to the Hodge-de Rham spectral triple of $\T^2$. The equivalence is given by the $L^2$-closure of the map $T$ in \eqref{eq:T}.
\end{rem}

For all $f$ we have
$$
[D,f]=L\circ T(\dd f),
$$
so that under the isomorphism $T$, Clifford multiplication becomes left matrix multiplication.
One easily computes $\Cl_D(A)$, which is given by the full matrix algebra $M_2(C(\T^2))$ acting via left multiplication on $H$. Indeed, let us denote by
\begin{equation}\label{eq:generators}
u(x,y)\coloneqq e^{ix}
\qquad\text{and}\qquad
v(x,y)\coloneqq e^{iy}
\end{equation}
the unitary generators of $A$. Since
$$
-u^*[D,u]=L_{\sigma_1} ,\qquad
-v^*[D,v]=L_{\sigma_2} ,
$$
the elements $\sigma_1$ and $\sigma_2$ belong to $\Cl_D(A)$, and as is well-known  they generate $M_2(\R)$ as an algebra. Elements in the norm-closure of $A$ 
belong to $Cl_D(A)$ as well, which therefore contains the algebra generated by $C(\T^2)$ and $M_2(\R)$, \emph{i.e.},\ all of $M_2(C(\T^2))$. 
We therefore have the following results.
\begin{lemma}\label{lemma:7}
$\Cl_D(A)=M_2(C(\T^2))$, which acts on $H$ by left matrix multiplication.
\end{lemma}
\begin{prop}\label{prop:8}
$\Cl_D(A)'=M_2(L^\infty(\T^2))$, which acts on $H$ by right matrix multiplication.
\end{prop}
\begin{proof}
Up to a natural identification, $H=L^2(\T^2)\otimes M_2(\R)$ and $\Cl_D(A)=C(\T^2)\otimes M_2(\R)$ where $M_2(\R)$ acts by left matrix multiplication. Continuous functions on $\T^2$ are dense (in the strong operator topology) in the von Neumann algebra $L^\infty(\T^2)$ of essentially bounded measurable functions (with respect to the Lebesgue measure), and $L^\infty(\T^2)$ is its own commutant.
Using the commutation theorem for tensor products of von Neumann algebras \cite{RvD75} one finds
$\Cl_D(A)'=(L^\infty(\T^2)\otimes M_2(\R))'=L^\infty(\T^2)'\otimes M_2(\R)'$ and the thesis follows.
\end{proof}

Recall \eqref{eq:C1} and \eqref{eq:C2}, the definitions of the 
the antilinear maps $C_1$ and $C_2$ on forms, the former multiplicative and the latter antimultiplicative, with ${C_1(\omega)=\overline{\omega}}$ the pointwise complex conjugate of a form and $C_2(\omega)=(-1)^{k(k-1)/2}\overline{\omega}$ on forms of degree $k$. 
On $H$ two corresponding antilinear isometries are given by $J_i\coloneqq \gamma TC_iT^{-1}$, ${i=1,2}$ (where $\gamma$ is included to simplify the expressions). One easily checks that, for all $a,b,c,d\in L^2(\T^2)$
\begin{equation}\label{eq:torusJ}
J_1\begin{pmatrix}a & b \\ c & d\end{pmatrix}=\begin{pmatrix}\overline{d} & \overline{c} \\ \overline{b} & \overline{a}\end{pmatrix}
,\qquad
J_2\begin{pmatrix}a & b \\ c & d\end{pmatrix}=\begin{pmatrix}\overline{a} & \overline{c} \\ \overline{b} & \overline{d}\end{pmatrix}
.
\end{equation}
Note that $J_2$ is just matrix Hermitian conjugation.
It is useful in the computations to recognise that
$J_1=L_{\sigma_1}R_{\sigma_1}J_0$, where $J_0$ is entrywise pointwise complex conjugation,
\begin{equation}\label{eq:torusCC}
J_0\begin{pmatrix}a & b \\ c & d\end{pmatrix}=\begin{pmatrix} \overline{a} & \overline{b} \\ \overline{c} & \overline{d} \end{pmatrix} .
\end{equation}
In particular, it follows that $J_1$ is multiplicative (since $J_0$ is multiplicative and $\sigma_1^2=1$):
$$
J_1(\alpha\beta)=J_1(\alpha)J_1(\beta)
$$
for all $\alpha\in M_2(L^\infty(\T^2)) $ and $\beta\in M_2(L^2(\T^2))$. By contrast, $J_2$ is antimultiplicative:
$$
J_2(\alpha\beta)=J_2(\beta)J_2(\alpha)
$$
for all $\alpha\in M_2(L^2(\T^2)) $ and $\beta\in M_2(L^\infty(\T^2))$. In particular, for all $m\in M_2(L^\infty(\T^2))$ we have
\begin{equation}\label{eq:LR}
J_0 L_m J_0 = L_{\overline{m}} ,\qquad
J_1 L_m J_1 = L_{J_1(m)} ,\qquad
J_2 L_m J_2 = R_{m^*} ,
\end{equation}
along with the analogous relations with $L$ and $R$ interchanged.

\begin{prop}\label{prop:9}
$(A,H,D,\gamma,J_1)$ is a unital even real spectral triple.
\end{prop}

\begin{proof}
A straightforward computation.
\end{proof}

\begin{prop}\label{prop:10}
$(A,H,D,\gamma,J_2)$ is a unital even spectral triple with $\tau$-twisted real structure, where $\tau=J_1J_2$, or explicitly,
\begin{equation}\label{eq:tauJ2}
\tau\begin{pmatrix} a & b \\ c & d \end{pmatrix}\coloneqq \begin{pmatrix} d & b \\ c & a \end{pmatrix} .
\end{equation}
Furthermore, this spectral triple satisfies the second-order condition and the Hodge condition \eqref{eq:Hodge}.
\end{prop}

\begin{proof}
The condition \eqref{eq:2ndorder} follows from the fact that conjugation by $J_2$ transforms left into right matrix multiplication, \emph{cf.}~\eqref{eq:LR}.
Since $J_1$ satisfies (\ref{eq:Jb}), clearly $J_2=\tau J_1=J_1\tau$ satisfies \eqref{eq:twistedJ}. The subspace $E\coloneqq M_2(C(\T^2))=\Cl_D(A)$ is dense in $H$ and a self-Morita equivalence $\Cl_D(A)$-bimodule, with the left/right action of $\Cl_D(A)$ on $E$ given by left/right multiplication, as requested. The map $L_{m}\mapsto J_2L_{m^*}J^{-1}_2=R_m$ transforms the left into the right action and vice versa, \emph{cf.}~\eqref{eq:LR}, so that the self-Morita equivalence is implemented by $J_2$.
\end{proof}

The spectral triples in Propositions \ref{prop:9} and \ref{prop:10} both have signs ${\varepsilon=\varepsilon'=\varepsilon''=+1}$, \emph{i.e.},~KO-dimension $0\bmod8$.
Notice that: (i) $\tau=TC_1C_2T^{-1}$, and $C_1C_2$ is the canonical anti-involution of the Clifford algebra, given on $k$-forms by $(-1)^{k(k-1)/2}$ times the identity; (ii) the spectral triple in Proposition \ref{prop:9} does \emph{not} satisfy the second-order condition,
for example because $J_1L_{\sigma_1}J_1=L_{\sigma_1}$ does not commute with $L_{\sigma_2}$; (iii) the spectral triple in Proposition \ref{prop:10} does \emph{not} satisfy condition (\ref{eq:Jb}), since
$$
-J_2DJ_2=R_{\sigma_1}i\frac{\partial}{\partial x}+R_{\sigma_2}i\frac{\partial}{\partial y}\neq \pm D .
$$

The next theorem shows that (\ref{eq:Jb}) and \eqref{eq:2ndorder} are incompatible, so that if one wants the second-order condition to be satisfied, one is forced to introduce a twist.

\begin{thm}\label{thm:main}
The spectral triple $(A,H,D)$ admits no antilinear isometry $J$ satisfying both \textup{(\ref{eq:Jb})} and \eqref{eq:2ndorder}.
\end{thm}

\begin{proof}
Assume that both (\ref{eq:Jb}) and \eqref{eq:2ndorder} are satisfied. Since conjugation by $J$ maps $a\in\Cl_D(A)$ into its commutant, it follows from \eqref{eq:2ndorder} that
\begin{equation}\label{eq:implicitly}
JL_{a^*}J^{-1}=R_{\phi(a)}
\end{equation}
for some $\phi(a)\in M_2(L^\infty(\T^2))$. This defines a $*$-homomorphism
$$
\phi\colon M_2(C(\T^2))\to M_2(L^\infty(\T^2)) .
$$
If $f\in C^\infty(\T^2)$ is a smooth scalar function,
it follows from (\ref{eq:Jb}) that, since both $J$ and $L_{f^*}$ preserve the domain of $D$, $R_{\phi(f)}$ preserves the domain of $D$ as well and
\begin{equation}\label{eq:DRL}
[D,R_{\phi(f)}]=\varepsilon' J[D,L_{f^*}]J^{-1}
\end{equation}
extends to a bounded operator on $H$. Moreover, if we apply $R_{\phi(f)}$ to $1\in H$ we deduce that $R_{\phi(f)}(1)=\phi(f)$ is in the domain of self-adjointness of $D$, \emph{i.e.},~a matrix of functions in the Sobolev space $W^{1,2}(\T^2)$.
Now let $u$ be the unitary element in \eqref{eq:generators} and write
$$
\phi(u)=f_0\sigma_0+f_1\sigma_1+f_2\sigma_2+f_3\sigma_3
$$
for some $f_0,\ldots,f_3\in W^{1,2}(\T^2)$. Then
\begin{align*}
0 \stackrel{\text{(\ref{eq:Jb})}}{=}\big[L_{\sigma_1},J[D,L_{u^*}]J^{-1}\big]&=
\varepsilon'\big[L_{\sigma_1},[D,R_{\phi(u)}]\big] \\
&=\varepsilon'\sum_{k=0}^3\big[L_{\sigma_1},[D,f_k]\big]R_{\sigma_k}=
-\varepsilon'\sum_{k=0}^3L_{\sigma_3}\frac{\partial f_k}{\partial y}R_{\sigma_k} .
\end{align*}
From the independence of the linear maps $R_{\sigma_k}$ we get $\partial_yf_k=0$ for all $k$ (where the derivative is in the sense of distributions), and in a similar way one proves that $\partial_x f_k=0$ for all $k$. Thus, $\phi(u)$ is a constant matrix and
$$
L_{\sigma_1}=[D,L_{u^*}]=\varepsilon' [D,R_{\phi(u)}]J^{-1}=0 ,
$$
which is a contradiction.
\end{proof}

One may wonder how unique the example in \eqref{prop:9} is, and how unique a $J$ satisfying the second-order condition is. A partial answer is provided by the following proposition.

\begin{prop}\label{prop:12}
Let $U\in M_2(L^\infty(\T^2))$ be a unitary operator and let $\tau$ be given by \eqref{eq:tauJ2}. Then:
\begin{itemize}
\item[(i)] The operator
$$
J_U\coloneqq L_{U}R_{U}J_2
$$
is an antilinear isometry satisfying $J_U^2=1$ and the second-order condition.

\item[(ii)] If $\tau(U)=U^*$ almost everywhere, then $J_U$ satisfies \eqref{eq:twistedJ} with $\varepsilon'=+1$ and twist given by
$$
\tau_U\coloneqq L_{U}R_{U}\tau .
$$
Moreover, in such a case $J_U$ and $\tau_U$ commute, and $\tau_U^2=1$.

\item[(iii)] If $\tau(U)=U^*$ almost everywhere, then $J_U$ also commutes with $\gamma$.

\end{itemize}
Thus if $\tau(U)=U^*$, the data $(A,D,H,\gamma,J_U)$ form a unital even real spectral triple with $\tau_U$-twisted real structure.
\end{prop}

\begin{proof}
(i) From \eqref{eq:LR} we deduce that $L_{U}R_{U}J_2=J_2L_{U^*}R_{U^*}$ and from the unitarity of $U$ it follows that $J_U^2=1$. It also follows that
$$
J_UL_mJ_U=J_2L_{U^*mU}J_2=R_{U^*mU} \in\Cl_D(A)'
$$
for all $m\in M_2(L^\infty(\T^2))$.

\smallskip

\noindent
(ii) Consider some $m\in H$. Notice that $\tau(m)=\sigma_1m^{\mathrm{t}}\sigma_1$, so that $\tau$ is antimultiplicative and
$$
\tau\, L_UR_U(m)=\tau(UmU)=\tau(U)\tau(m)\tau(U)
$$
for all $m\in H$, which means that
$$
\tau\,L_UR_U=L_{\tau(U)}R_{\tau(U)}\tau .
$$
From the proof of part (i), we have $J_U\tau_U=J_2\tau$. We now compute
$$
\tau_UJ_U=L_UR_U\tau\,L_UR_UJ_2=L_{U\sigma(U)}R_{U\sigma(U)}\tau J_2 .
$$
If $\sigma(U)=U^*$ then $J_U\tau_U=J_2\tau=\tau J_2=\tau_UJ_U$ and
$$
\tau_UJ_UD=\tau J_2D=DJ_2\tau=DJ_U\tau_U.
$$
One similarly checks that $\tau_U^2=1$. Since $J_U\tau_U=J_2\tau$, the operator clearly preserves the domain of $D$.\footnote{This holds even if the matrix entries of $U$ are not necessarily in the domain of $D$.}

\smallskip

\noindent
(iii) For almost all $p\in\T^2$, $U(p)$ is a constant unitary matrix and the compatibility condition with $\tau$ implies that
$$
U(p)=\begin{pmatrix} a & b \\ \overline{b} & \overline{a} \end{pmatrix}
$$
for some $a,b\in\C$. Such a matrix is unitary if and only if it is of the form
\begin{equation}\label{eq:twotypes}
U(p)=\begin{pmatrix} e^{i\theta} & 0 \\ 0 & e^{-i\theta} \end{pmatrix}
\qquad\text{or}\qquad
U(p)=\begin{pmatrix} 0 & e^{i\theta} \\ e^{-i\theta} & 0 \end{pmatrix}
\end{equation}
for some $\theta\in\R$. In the first case $U(p)$ commutes with $\sigma_3$, while in the second they anticommute. In both cases $L_{U(p)}R_{U(p)}$ commutes with $\gamma$.
\end{proof}

In Proposition \ref{prop:12}(ii) the condition $\tau(U)=-U^*$ would work as well, but notice that it implies $\tau(iU)=(iU)^*$ and the rescaling $U\mapsto iU$ simply changes $J_U$ and $\tau_U$ by a sign.
The condition $\tau(U)=U^*$ implies that, at almost every point $p$, $U$ is of one of the two types in \eqref{eq:twotypes}, though both $\theta$ and the type may depend on $p$
($U$ being a constant matrix is a special case).

\medskip

Let us close this section with a comment on orientability. Note that $\Omega_D^1(A)\neq 0$, since it is isomorphic to the $A$-module of de Rham forms on the torus. The spectral triple in Proposition \ref{prop:10} is then not orientable, since it satisfies the second-order condition (see the remark at the end of \S\ref{sec:Morita}). But the triple in Proposition \ref{prop:9} is also not orientable: $\gamma\notin\Cl_D(A)$ since, for example, it doesn't commute with $R_{\sigma_1}\in\Cl_D(A)'$. To get an orientable spectral triple we need to choose a different grading operator.

\begin{prop}
The unital even real spectral triple $(A,H,D,L_{\sigma_3},J_1)$ is orientable.
\end{prop}

\begin{proof}
A straightforward computation. A possible choice of Hochschild $2$-cycle giving the orientation is
$$
c=-2^{-1}iu^*v^*\otimes (u\otimes v-v\otimes u) .
$$
Clearly
$$
-2^{-1}iu^*v^*\big([D,u][D,v]-[D,v][D,u]\big)=L_{\sigma_3}
$$
is the grading, and one can check that the Hochschild boundary of $c$ is zero with a simple computation.
\end{proof}

One may wonder what this new grading is. Let $\chi\coloneqq T^{-1}\circ L_{\sigma_3}\circ T$. Then one finds
$$
\chi(1)=idx\wedge dy ,\qquad
\chi(dx)=idy ,\qquad
\chi(dy)=-idx ,\qquad
\chi(dx\wedge dy)=-i .
$$
Evidently the map $\chi$ is the grading coming from the Hodge star operator.

\section{Products of spectral triples}\label{sec:products}

Let $(A_1,H_1,D_1,\gamma_1,J_1)$ and $(A_2,H_2,D_2,J_2)$ be
two unital real spectral triples, the former even. We define their product $(A,H,D,J)$ by
\begin{equation}\label{eq:product}
A=A_1\otimes A_2 ,\quad
H=H_1\otimes H_2 ,\quad
D=D_1\otimes 1+\gamma_1\otimes D_2 ,\quad
J=J_1\otimes J_2 .
\end{equation}
Here we assume that $A_1$ and $A_2$ are both complex, so that $\otimes$ is everywhere the tensor product over $\C$ (algebraic, minimal or of Hilbert spaces depending on the type of object we are considering).
If both spectral triples are even, a grading on the product is given by:$$
\gamma=\gamma_1\otimes\gamma_2 .
$$
We will not consider the case where both spectral triples are odd. It is not very different, but for the sake of brevity we will always assume that at least one of the spectral triples is even. Notice that the example we are interested in, the Hodge-de Rham spectral triple of a closed oriented Riemannian manifold, is always even.

If $J_1$ and $J_2$ satisfy \eqref{eq:twistedJ} for some twists $\tau_1$ and $\tau_2$, then $J$ satisfies \eqref{eq:twistedJ} with twist $\tau=\tau_1\otimes\tau_2$.\footnote{For suitable values of the signs $\varepsilon,\varepsilon',\varepsilon''$, which will not be discussed here. See the comments in \S\ref{sec:2ndproduct}.}

\begin{lemma}\label{lemma:13}
$\Omega^1_D(A)=\Omega^1_{D_1}(A_1)\otimes A_2+\gamma_1A_1\otimes\Omega^1_{D_2}(A_2)$.
\end{lemma}

\begin{proof}
Elements of the algebraic tensor product $A_1\otimes A_2$ are finite sums of decomposable tensors. From the fact that
\begin{equation}\label{eq:15}
(a_1\otimes a_2)[D,b_1\otimes b_2]=a_1[D_1,b_1]\otimes a_2b_2+\gamma_1a_1b_1\otimes a_2[D_2,b_2]
\end{equation}
for all $a_1,b_1\in A_1$ and $a_2,b_2\in A_2$,
we get the inclusion
$$
\Omega^1_D(A)\subset\Omega^1_{D_1}(A_1)\otimes A_2+\gamma_1A_1\otimes\Omega^1_{D_2}(A_2) .
$$
Since $A_1$ and $A_2$ are unital, if we choose $b_1=1$ in \eqref{eq:15} we find that $\gamma_1A_1\otimes\Omega^1_{D_2}(A_2)$ is a vector subspace of $\Omega^1_D(A)$; if we choose $b_2=1$ (and $b_1$ arbitrary), we find that $\Omega^1_{D_1}(A_1)\otimes A_2$ is a vector subspace of $\Omega^1_D(A)$.
\end{proof}

\noindent It then follows that
\begin{equation}\label{eq:16}
\Cl_D(A)\subset
\Cl^{\gamma_1}_{D_1}(A_1)\otimes\Cl_{D_2}(A_2).
\end{equation}

\begin{lemma}\label{lemma:2}
If $\gamma_1\in\Cl_{D_1}(A_1)$, then
\begin{equation}\label{eq:cliffordproduct}
\Cl_D(A)=\Cl_{D_1}(A_1)\otimes\Cl_{D_2}(A_2).
\end{equation}
\end{lemma}

\begin{proof}
Since $\Cl^{\gamma_1}_{D_1}(A_1)=\Cl_{D_1}(A_1)$,
the inclusion ``$\subset$'' follows from \eqref{eq:16}.

We saw in the proof of Lemma \ref{lemma:13} that
$\Omega^1_{D_1}(A_1)\otimes A_2$ and 
$\gamma_1A_1\otimes\Omega^1_{D_2}(A_2)$
are contained in $\Omega^1_D(A)$, and therefore in $\Cl_D(A)$.

Since $A_2$ is unital, both
$\Omega^1_{D_1}(A_1)\otimes 1$ and $A_1\otimes 1\subset A$ are in $\Cl_D(A)$. Thus $\Cl_D(A)\supset\Cl_{D_1}(A_1)\otimes 1$.

Since $A_1$ is unital, both
$\gamma_1\otimes\Omega^1_{D_2}(A_2)$ and $1\otimes A_2$ are in $\Cl_D(A)$. But $\gamma_1\otimes 1\in \Cl_{D_1}(A_1)\otimes 1\subset\Cl_D(A)$ as well. Therefore $1\otimes\Omega^1_{D_2}(A_2)=(\gamma_1\otimes 1)\big(\gamma_1\otimes\Omega^1_{D_2}(A_2)\big)$ is contained in $\Cl_D(A)$ as well. This proves that $\Cl_D(A)\supset 1\otimes\Cl_{D_2}(A_2)$ and hence the inclusion
$\Cl_D(A)\supset\Cl_{D_1}(A_1)\otimes\Cl_{D_2}(A_2)$.
\end{proof}

\subsection{Products and the spin condition}
Given two spectral triples satisfying one of the conditions in Def. \ref{df:5}, one wonders if the product satisfies such a condition as well. The answer is affirmative for condition \eqref{eq:spin}.

\begin{prop}[Spin]
If two unital real spectral triples satisfy \eqref{eq:spin}, then their product satisfies \eqref{eq:spin} as well.
\end{prop}

\begin{proof}
Using the notation above, suppose $E_1\subset H_1$ and $E_2\subset H_2$ are dense subspaces, with $E_1$ and $E_2$ full right Hilbert $\overline{A}_1$- and $\overline{A}_2$-modules respectively, where the right action of $\xi\in\overline{A}_i$ is given by $J_i\xi^*J_i^{-1}$, and suppose for $i=1,2$ one has
$$
\Cl_{D_i}(A_i)=\mathcal{K}_{\overline{A}_i}(E_i) .
$$
Let $E\coloneqq E_1\otimes E_2$ be the exterior tensor product of Hilbert modules. This is a full right Hilbert $\overline{A}_1\otimes\overline{A}_2$-module, where here the tensor product is the minimal tensor product of $C^*$-algebras. Note that $\overline{A}_1\otimes\overline{A}_2=\overline{A}$. The right action of a decomposable tensor $\xi=\xi_1\otimes\xi_2\in\overline{A}$ is given by 
$J\xi^*J^{-1}=J_1\xi_1^*J_1^{-1}\otimes J_2\xi_2^*J_2^{-1}$, so by linearity and continuity the right action of any element $\overline{A}$ is implemented by $J=J_1\otimes J_2$. One also has has
$$
\Cl_{D_1}(A_1)\otimes\Cl_{D_2}(A_2)=\mathcal{K}_{\overline{A}_1}(E_1)\otimes\mathcal{K}_{\overline{A}_2}(E_2) .
$$
But $\mathcal{K}_{\overline{A}_1}(E_1)\otimes\mathcal{K}_{\overline{A}_2}(E_2)=\mathcal{K}_{\overline{A}}(E)$ (see \emph{e.g.},~\cite[p.~45]{Lan95}). From Proposition \ref{prop:6}(i) it follows that $\gamma_1\in\Cl_D(A)$. From Lemma \ref{lemma:2} it follows that $\Cl_{D_1}(A_1)\otimes\Cl_{D_2}(A_2)=\Cl_D(A)$. Hence $\Cl_D(A)=\mathcal{K}_{\overline{A}}(E)$ as requested, and the product spectral triple satisfies \eqref{eq:spin}.
\end{proof}

Recall that if an even spectral triple is spin then it is also even-spin. In the next example we present two even-spin spectral triples whose product is not even-spin (and then also not spin).

\begin{ex}[Even-spin]\label{ex:evenspin}
Let $A_1=M_2(\C)$ be represented on $H_1=M_2(\C)\otimes\C^2$ by left multiplication on the first factor (we think of elements of $\C^2$ as column vectors), and $J_1(a\otimes v)\coloneqq a^*\otimes\overline{v}$ where $a^*$ is the Hermitian conjugate and $\overline{v}$ the componentwise conjugation. Also let
$$
D_1(a\otimes v)\coloneqq [\sigma_1,a]\otimes\sigma_1v ,\qquad
\gamma_1\coloneqq 1\otimes\sigma_3 .
$$
Since for $a=-i\sigma_3$ and $b=\sigma_2\in A_1$ one has
\begin{equation}\label{eq:2}
a[D_1,b]=1\otimes \sigma_1 =: \omega ,
\end{equation}
and $1$-forms are freely generated (as an $A_1$-module) by $\omega$.
The Clifford algebra $\Cl_{D_1}^{\gamma_1}(A_1)$ is then generated by $M_2(\C)\otimes 1$, $\omega=1\otimes\sigma_1$, $\gamma_1=1\otimes\sigma_3$. Hence
$$
\Cl_{D_1}^{\gamma_1}(A_1)=M_2(\C)\otimes M_2(\C)
$$
with its action on $H_1$ given by left multiplication. The commutant is $\Cl_{D_1}^{\gamma_1}(A_1)'=J_1A_1J_1^{-1}$ and \eqref{eq:evenspin} is satisfied.
Note that $\gamma_1\notin\Cl_{D_1}(A_1)$, so \eqref{eq:spin} is not satisfied.

Let $(A,H,D,\gamma,J)$ be a product of two copies of the above spectral triple. $\Cl_D^\gamma(A)$ is generated by $A_1$, $A_2$, $\gamma_1\otimes\gamma_2=1\otimes\sigma_3\otimes 1\otimes\sigma_3$, and the $1$-forms
$$
\omega\otimes 1=1\otimes\sigma_1\otimes 1\otimes 1 ,\qquad
\gamma_1\otimes\omega=1\otimes\sigma_3\otimes 1\otimes \sigma_1 .
$$
The element
$$
1\otimes\sigma_1\otimes 1\otimes \sigma_2
$$
is in the commutant of $\Cl_D^\gamma(A)$, but it doesn't belong to $JAJ^{-1}$.
\end{ex}

Let us make note of this result.

\begin{rem}
There exist spectral triples satisfying \eqref{eq:evenspin} whose product does not satisfy \eqref{eq:evenspin} (nor \eqref{eq:spin}).
\end{rem}

Finally, let us consider a mixed case. In the following we present an even-spin spectral triple and a spin spectral triple (which is in fact also Hodge), whose product is not even-spin (and indeed also not spin).

\begin{ex}[Mixed]
Take the first spectral triple to be as in Example \ref{ex:evenspin} and let the second one be given by $A_2=H_2=M_2(\C)$, $D_2(a)=[\sigma_1,a]$, and $J_2(a)=a^*$ for all $a\in A_2$. This spectral triple satisfies both \eqref{eq:spin} and \eqref{eq:Hodge} (which is only possible in the finite-dimensional case when $1$-forms are contained in the algebra, $\Omega^1_{D_2}(A_2)\subset A_2$, which implies $A_2=\Cl_{D_2}(A)$).

The Clifford algebra $\Cl_D^\gamma(A)$ of the product triple is generated by $A_1\otimes 1,1\otimes A_2$ and the $1$-form $\omega\otimes 1=1\otimes\sigma_1\otimes 1$ (where $\omega$ is the element \eqref{eq:2}); the latter element belongs to the commutant of $\Cl_D^\gamma(A)$, but not to $A_1\otimes A_2$, hence the product spectral triple does not satisfy \eqref{eq:evenspin}.
\end{ex}

Mixed Hodge-spin cases are not particularly interesting. There is no reason to expect that such a product is either Hodge or spin
-- there is also no reason to expect that a product of two Hodge spectral triples is spin, or that a product of two (even-)spin spectral triples is Hodge -- and in fact it is quite easy to produce counterexamples.

\subsection{Products and the second-order condition}\label{sec:2ndproduct}
Given two real spectral triples, one may wonder how the KO-dimension of their product is related to those of its two factors. It would be nice if the KO-dimension were multiplicative, but unfortunately if the real structure is defined by $J=J_1\otimes J_2$ this is not true. This was first noticed in \cite{Van99}, where a modified definition of $J$ was proposed in order to fix this problem (taking either $J=J_1\otimes J_2\gamma_2$ or $J=J_1\gamma_1\otimes J_2$ depending on the dimension of the factors). This study was completed in \cite{DD11}, where the odd-odd case was considered as well (in \cite{Van99} one of the spectral triples is always assumed to be even), along with several possible choices of Dirac operators and real structures. The modified definition of $J$, which perhaps seems somewhat artificial, was reinterpreted in \cite{Far17} as a graded tensor product.

Here we follow this idea in spirit, but we will find that the ``correct'' definition of $J$ is not the one in \cite{DD11,Far17,Van99}. Our motivation is that we want the second-order property \eqref{eq:2ndorder} to be preserved by products, and this will lead to yet another different definition of $J$. Although the natural way to study products of real spectral triples is in the category of graded vector spaces, we will argue that in terms of ``ungraded'' objects and operations this amounts to merely changing the real structure.

\medskip

We will use the same notation as the previous section and assume that we have two unital real spectral triples $(A_i,H_i,D_i,\gamma_i,J_i)$, $i=1,2$. For simplicity we will also assume that both spectral triples are even.\footnote{Recall that the Hodge-de Rham spectral triple of a closed oriented Riemannian manifold is always even, regardless of the dimension of the manifold, whether even or odd.}
If $v\in H_1$ is an eigenvector of $\gamma_1$ with eigenvalue $(-1)^{|v|}$ we will say that $v$ is homogeneous of degree $|v|$. Explicitly,
$$
|v|=\bigg\{\begin{array}{ll}
0 &\text{if}\;\gamma_1(v)=+v, \\[1pt]
1 &\text{if}\;\gamma_1(v)=-v.
\end{array} 
$$
A bounded operator $a\in\mathcal{B}(H_1)$ has degree $0$ (called \emph{even}) if it commutes with $\gamma_1$, and degree $1$ (called \emph{odd}) if it anticommutes with it. This notion extends to unbounded operators (such as $D_1$), provided that $\gamma_1$ preserves their domain.

The same definitions apply to the second spectral triple. 
According to Koszul's rule of signs, the graded tensor product $a\odot b\in\mathcal{B}(H)$ of two (homogeneous) bounded operators is now defined by
$$
(a\odot b)(v\otimes w)\coloneqq (-1)^{|b|\,|v|}av\otimes bw
$$
for all homogeneous vectors $v\in H_1$ and $w\in H_2$. This definition also makes sense when one of the two operators is unbounded. For example, if $a$ is unbounded, then $a\odot b$ will be unbounded with its domain given by the algebraic tensor product of $\operatorname{Dom}(a)$ and $H_1$.

If $b$ is odd, $a\odot b=a\gamma_1\otimes b$. Thus, the Dirac operator in a product of spectral triples can be written as (the closure of)
$$
D=D_1\odot 1+1\odot D_2 .
$$
Since we are considering unital spectral triples, the following lemma now becomes evident.

\begin{lemma}\label{lemma:21b}
$\Cl_D(A)=\Cl_{D_1}(A_1)\odot\Cl_{D_2}(A_2)$.
\end{lemma}

Here if $B_1\subset\mathcal{B}(H_1)$ and $B_2\subset\mathcal{B}(H_2)$ are $C^*$-subalgebras, we define 
$B_1\odot B_2$ as the norm closure in $\mathcal{B}(H)$ of the vector subspace spanned by elements $a\odot b$, with $a\in\mathcal{B}(H_1)$ and $b\in\mathcal{B}(H_2)$.

\begin{proof}
The inclusion ``$\subset$'' is given by \eqref{eq:16}. The opposite inclusion is analogous to the proof of Lemma \ref{lemma:2}: one shows that $A_1\otimes 1=A_1\odot 1$, $1\otimes A_2=1\odot A_2$,
$\Omega^1_{D_1}(A_1)\otimes 1=\Omega^1_{D_1}(A_1)\odot 1$
and $\gamma_1\otimes\Omega^1_{D_2}(A_2)=\gamma_1\odot\Omega^1_{D_2}(A_2)$ are contained in $\Cl_D(A)$, hence the thesis.
\end{proof}

The idea is now to modify the definition of the product real structure. Since this should in principle change the order of factors in a (tensor) product, what is suggested again by Koszul's rule of signs is to define $J$ by
\begin{equation}\label{eq:newJ}
J(v\otimes w) \coloneqq  (-1)^{|v|\hspace{1pt}|w|}J_1(v)\otimes J_2(w)
\end{equation}
for all homogeneous $v\in H_1$ and $w\in H_2$.\footnote{Note that this is exactly what happens in a product of Hodge-de Rham spectral triples of a manifold, if the real structure is the one coming from the main anti-involution of the Clifford algebra.}
With this choice, we have the following proposition.

\begin{prop}\label{prop:22}
Consider two unital even spectral triples with (possibly twisted) real structure, both satisfying the second-order condition. Then their product $(A,H,D,\gamma)$, equipped with the antilinear map in \eqref{eq:newJ}, satisfies the second-order condition as well.
\end{prop}

\begin{proof}
For $i=1,2$, let $a_i\in A_i$ and $\omega_i\in\Omega^1_{D_i}(A_i)$. The algebra $\Cl_D(A)$ is generated by $A_1\otimes 1$, $1\otimes A_2$ and elements of the form $\omega_1\otimes 1$ and $\gamma_1\otimes\omega_2$. Thus $\Cl_D(A)^\circ$ is generated by
\begin{equation}\label{eq:leftright}
\begin{aligned}
J(a_1^*\otimes 1)J^{-1} &= a_1^\circ\otimes 1, \\
J(\omega_1^*\otimes 1)J^{-1} &= \omega_1^\circ\otimes \gamma_2, \\
J(1\otimes a_2)J^{-1} &= 1\otimes a_2^\circ, \\
J(\gamma_1\otimes\omega_2)J^{-1} &= 1\otimes \omega_2^\circ.
\end{aligned}
\end{equation}
Note the presence of $\gamma_2$, while $\gamma_1$ has disappeared.
If the two factors satisfy \eqref{eq:2ndorder}, then the above four elements commute with $\Cl_D(A)$, and in particular,
$$
[\omega_1^\circ\otimes \gamma_2,
\gamma_1\otimes\omega_2]=0
$$
because $\omega_1^\circ$ anticommutes with $\gamma_1$ and $\gamma_2$ anticommutes with $\omega_2$. Thus the product spectral triple satisfies $\Cl_D(A)^\circ\subset\Cl_D(A)'$.
\end{proof}

We stress again that \eqref{eq:newJ} is \emph{not} the real structure in \cite{DD11,Far17,Van99}, and that Prop.~\ref{prop:22} does \emph{not} hold if the product real structure is defined like in \cite{DD11,Far17,Van99}.

Notice that in Prop.~\ref{prop:22} we do not claim that the product spectral triple is real. If one checks the conditions \eqref{eq:J} for $J$ one finds that there is a problem with (\ref{eq:Ja}) and (\ref{eq:Jb}). Since here we are mainly interested in the second-order and Hodge condition, we will not investigate how to modify \eqref{eq:newJ} so that also (\ref{eq:Ja}) and (\ref{eq:Jb}) are also satisfied. We will merely make the following observation.

\begin{prop}
If the spectral triples $(A_i,H_i,D_i,\gamma_i,J_i)$, $i=1,2$, satisfy \textup{(\ref{eq:Jc})} with sign $\varepsilon''_i$, then their product -- with the product real structure given by \eqref{eq:newJ} -- satisfies \textup{(\ref{eq:Jc})} with sign $\varepsilon''=\varepsilon_1''\varepsilon_2''$.

If in addition the factors satisfy \textup{(\ref{eq:Ja})} with sign $\varepsilon_i$ and $\varepsilon_1''=\varepsilon_2''=+1$, then their product will satisfy \textup{(\ref{eq:Ja})} with sign $\varepsilon=\varepsilon_1\varepsilon_2$.\footnote{Notice that this is the case of the Hodge-de Rham spectral triples. It happens when the KO-dimension is a multiple of $4\bmod8$.}
\end{prop}

\begin{proof}
On decomposable homogeneous tensors
\begin{align*}
J\gamma(v\otimes w) &=(-1)^{|v|\hspace{1pt}|w|+|v|+|w|}J_1(v)\otimes J_2(w) , \\
\gamma J(v\otimes w) &=\varepsilon''_1\varepsilon''_2(-1)^{|v|\hspace{1pt}|w|+|v|+|w|}J_1(v)\otimes J_2(w) ,
\end{align*}
which proves the first part of the statement.

If $\varepsilon''_1=\varepsilon''_2=+1$, $J_i$ does not change the degree of a vector, and one easily verifies that
$$
J^2(v\otimes w)=(-1)^{2|v|\hspace{1pt}|w|}J_1^2(v)\otimes J_2^2(w)
$$
on decomposable homogeneous tensors. Hence $J^2=J_1^2\otimes J_2^2$ and we get the second part of the theorem.
\end{proof}

The problem with condition (\ref{eq:Jb}) is not surprising, since in the example of a closed oriented Riemaniann manifold we are forced to introduce a twist to make it work. Note that we can introduce another graded product $\odot'$ via the rule
$$
(a\odot' b)(v\otimes w)=(-1)^{|a|\,|w|}av\otimes bw
$$
for all homogeneous $v\in H_1$ and $w\in H_2$ and all homogeneous operators $a,b$ on $H_1,H_2$. With this convention, $a\odot' b=a\otimes b\gamma_2$ for all $a$ of degree $1$. This is the natural convention for right modules, \emph{i.e.},~if we imagine that endomorphisms act from the right on vectors. This graded product gives an alternative Dirac operator on $H_1\otimes H_2$:
\begin{equation}\label{eq:secondDirac}
\widetilde{D}\coloneqq D_1\odot' 1+1\odot' D_2=
D_1\otimes\gamma_2+1\otimes D_2 .
\end{equation}
It turns out that the modified real structure $J$ transforms the ``left'' into the ``right'' Dirac operator.

\begin{prop}
If the spectral triples $(A_i,H_i,D_i,\gamma_i,J_i)$, $i=1,2$, satisfy \eqref{eq:J} with signs $(\varepsilon_i,\varepsilon'_i,\varepsilon''_i)$ and
$$
\varepsilon'_1\varepsilon''_1=\varepsilon'_2
$$
then their product -- with $J$ given by \eqref{eq:newJ} -- satisfies
$$
JD=\varepsilon'\widetilde{D}J
$$
with $\widetilde{D}$ as in \eqref{eq:secondDirac} and $\varepsilon'=\varepsilon'_1$.
\end{prop}

\begin{proof}
A straightforward computation.
\end{proof}

The following observation will be useful later on, and holds regardless of the signs in \eqref{eq:J}.
If $B_1\subset\mathcal{B}(H_1)$ and $B_2\subset\mathcal{B}(H_2)$ are $C^*$-subalgebras, denote by
$B_1\odot' B_2$ the norm closure in $\mathcal{B}(H)$ of the vector subspace spanned by elements $a\odot' b$, with $a\in\mathcal{B}(H_1)$ and $b\in\mathcal{B}(H_2)$.

\begin{lemma}\label{lemma:25}
In a product of unital real even spectral triples, and with $J$ given by \eqref{eq:newJ}, one has
$J\Cl_D(A)J^{-1}=\Cl_{D_1}(A_1)\odot'\Cl_{D_2}(A_2)$.
\end{lemma}

\begin{proof}
We see from \eqref{eq:leftright} that conjugation by $J$ sends generators of $\Cl_D(A)$ into generators of $\Cl_{D_1}(A_1)\odot'\Cl_{D_2}(A_2)$.
\end{proof}

\subsection{Products and the Hodge condition}\label{sec:prodHodge}
The behaviour of the Hodge condition under products is more technical and to simplify the discussion we will study it only in the finite-dimensional case.
We want to prove the following proposition.

\begin{prop}\label{prop:Hodgeproduct}
Let $(A_i,H_i,D_i,\gamma_i,J_i)$ be two unital finite-dimensional even real spectral triples satisfying the Hodge condition \eqref{eq:Hodge}. Define $(A,H,D)$ as in \eqref{eq:product} and $J$ as in \eqref{eq:newJ}. Then $J\Cl_D(A)J^{-1}=\Cl_D(A)'$.
\end{prop}

\noindent
That is, that the product spectral triple satisfies the Hodge condition as well. In fact, in light of Lemmas \ref{lemma:21b}
and \ref{lemma:25}, Prop.~\ref{prop:Hodgeproduct} is a corollary of the following theorem.

\begin{thm}\label{thm:gct}
Let $B_i\subset\operatorname{End}_{\C}(H_i)$ be two unital subalgebras, $i=1,2$. Then
$$
(B_1\odot B_2)'=B_1'\odot'B_2' .
$$
\end{thm}

\begin{proof}
For all homogeneous elements $a\in B_1$, $b\in B_2$, $c\in B_1'$, $d\in B_2'$, $v\in H_1$, and $w\in H_2$ one has:
\begin{align*}
(a\odot b)(c \odot' d)(v\otimes w)&=(-1)^{|b||v|+|c||w|+|b||c|}acv\otimes bdw \\
&=(c\odot' d)(a \odot b)(v\otimes w) .
\end{align*}
This proves the inclusion $B_1'\odot' B_2'\subseteq (B_1\odot B_2)'$. For the opposite inclusion, since all spaces are finite-dimensional and $\gamma_2$ is invertible,
every element $\omega\in\operatorname{End}_{\C}(H)$ can be written as a finite sum
\begin{align*}
\omega &=\sum\nolimits_i (a_i\otimes b_i+c_i\otimes d_i\gamma_2) \\
&=\sum\nolimits_i (a_i\odot' b_i+c_i\odot' d_i)
\end{align*}
for some $a_i,c_i\in\operatorname{End}_{\C}(H_1)$,
$b_i,d_i\in\operatorname{End}_{\C}(H_2)$, $a_i$ even and $c_i$ odd. The elements $\{b_i,d_j\}$ can be chosen to be linearly independent.
Assume $\omega\in (B_1\odot B_2)'$. Since $B_2$ is unital, for all $x\in B_1$ one has $x\otimes 1=x\odot 1\in B_1\odot B_2$ and
$$
[\omega,x\otimes 1]=\sum\nolimits_i ([a_i,x]\otimes b_i+[c_i,x]\otimes d_i\gamma_2)
$$
must be zero. From the linear independence of the elements in the second factor, we deduce $[a_i,x]=[c_i,x]=0$, so that $a_i,c_i\in B_1'$ and $\omega\in B_1'\odot'\operatorname{End}_{\C}(H_2)$. It follows that $\omega$ can be written as a finite sum 
$\omega=\sum\nolimits_i (\widetilde{a}_i\odot' \widetilde{b}_i+\widetilde{c}_i\odot' \widetilde{d}_i)$ where 
$\widetilde{a}_i,\widetilde{c}_i\in B_1'$ and $\widetilde{b}_i,\widetilde{d}_i\in \operatorname{End}_{\C}(H_2)$, $\widetilde{a}_i$ even, $\widetilde{c}_i$ odd and now
$\{\widetilde{a}_i,\widetilde{c}_i\}$ are chosen to be linearly independent.
Since $B_1$ is unital, for all even $y\in B_2$ one has $1\odot y=1\otimes y\in B_1\odot B_2$ and for all even $y\in B_2$ one has $1\odot y=\gamma_1\otimes y\in B_1\odot B_2$. If $y$ is even,
$$
[\omega,1\odot y]=\sum\nolimits_i (\widetilde{a}_i\otimes[\widetilde{b}_i,y]+\widetilde{c}_i\otimes [\widetilde{d}_i,y]\gamma_2),
$$
while if $y$ is odd,
$$
[\omega,1\odot y]=\sum\nolimits_i (\widetilde{a}_i\gamma_1\otimes [\widetilde{b}_i,y]-\widetilde{c}_i\gamma_1\otimes[\widetilde{d}_i,y]\gamma_2) .
$$
From the linear independence of the elements in the first factor we deduce that in both cases $[\widetilde{b}_i,y]=[\widetilde{d}_i,y]=0$, so that $\widetilde{b}_i,\widetilde{d}_i\in B_2'$ and $\omega\in B_1'\odot' B_2'$.
\end{proof}

\section*{Acknowledgements}
L.D.~and F.D'A.~thank the Fields Institute, Toronto, for hospitality. A.M.~thanks A.Rubin for helpful discussions.
This work was partially supported by the grant H2020-MSCA-RISE-2015-691246-QUANTUM DYNAMICS.

\end{document}